\documentclass{amsart}
\usepackage{latexsym}
\usepackage{amssymb}
\usepackage{amsmath}
\usepackage{amsfonts}

\def\R{\mathbb{R}}

\newcommand{\bmat}{\left[\begin{matrix}}
\newcommand{\emat}{\end{matrix}\right]}
\usepackage[latin1]{inputenc} 
\usepackage{amsthm}
\usepackage{amsxtra}
\usepackage{amscd}
\usepackage{setspace}   % enables line spacing commands 
\usepackage{mathrsfs}   % enables \mathscr
\usepackage{enumerate}

\newtheorem{theorem}{Theorem}
\newtheorem{proposition}[theorem]{Proposition}
\newtheorem{lemma}[theorem]{Lemma}

\newtheorem*{corollary}{Corollary}

\theoremstyle{remark}
\newtheorem*{remark}{Remark}

\theoremstyle{definition}

\newcommand{\Z}{\mathbb{Z}}
\newcommand{\Q}{\mathbb{Q}}

\newcommand{\Gr}{\mathrm{Gr}}
\newcommand{\F}{\mathbb{F}}
\newcommand{\rk}{\mathrm{rk}\,}
\newcommand{\SO}{\mathrm{SO}}

\newcommand{\SL}{\mathrm{SL}}

\newcommand{\spn}{\mathrm{span}}

    \title{Mean value formulas on sublattices and flags of the random lattice}
    \author{Seungki Kim}

\begin{document}

\maketitle

\begin{abstract}
We present extensions of the Siegel integral formula (\cite{Sie}), which counts the vectors of the random lattice, to the context of counting its sublattices and flags. Perhaps surprisingly, it turns out that many quantities of interest diverge to infinity.
\end{abstract}

\section{Introduction}

We start by recalling the celebrated Siegel integral formula (\cite{Sie}), one of the cornerstones of geometry of numbers. Let $X_n = \SL(n, \Z) \backslash \SL(n, \R)$ be the space of lattices of determinant $1$, and equip $X_n$ with the measure $\mu_n$ (defined up to a constant, which is to be determined by Theorem \ref{thm:siegel_int} below) that is inherited from the Haar measure of $\SL(n, \R)$; in particular, $\mu_n$ is invariant under the right $\SL(n, \R)$-action on $X_n$. In this setting, Siegel proved the following theorem.

\begin{theorem}[Siegel \cite{Sie}] \label{thm:siegel_int}
$\mu_n(X_n) < \infty$, so upon normalizing we may suppose $\mu_n(X_n) = 1$. Also, for $f: \R^n \rightarrow \R$ a compactly supported and bounded Borel measurable function, we have
\begin{equation*}
\int_{X_n} \sum_{x \in L \atop x \neq 0} f(x) d\mu_n(L) = \int_{\R^n} f(x) dx.
\end{equation*}
\end{theorem}

There are many other useful variations of Theorem \ref{thm:siegel_int}. For instance, Rogers proved the following, known as the Rogers integral formula.

\begin{theorem}[Rogers \cite{Rogers}] \label{thm:rogers_int}
Let $1 \leq k \leq n-1$. Then for $f: (\R^n)^k \rightarrow \R$ a compactly supported and bounded Borel measurable function, we have
\begin{equation*}
\int_{X_n} \sum_{x_1, \ldots, x_k \in L \atop \rk\langle x_1, \ldots, x_k \rangle = k} f(x_1, \ldots, x_k) d\mu_{n}(L) = \int_{\R^n} \ldots \int_{\R^n} f(x_1, \ldots, x_k) dx_1 \ldots dx_k.
\end{equation*}
\end{theorem}

Rogers also proved explicit formulas of this kind for the cases $\rk\langle x_1, \ldots, x_k \rangle = l$ for any $1 \leq l < k$ --- see (\cite{Rogers}) for the precise statement. This result, called the Rogers integral formula, is the essential tool for the study of random lattice vectors, since it makes possible to study the higher moments of the lattice vector-counting function $\sum_{x\in L \backslash \{0\}} f(x)$. See Schmidt (\cite{Sch3}) for yet more variants of the Siegel integral formula, and Kim (\cite{Kim}), Shapira and Weiss (\cite{SW}), and S\"odergren and Str\"ombergsson (\cite{SS}) for some of the recent applications of these mean value theorems.

The motivation of the present paper is to explore the extensions of these results from the counting of lattice vectors to that of rank $d < n$ sublattices and of flags. In other words, we ask whether they too demonstrate any interesting statistical behavior, like the lattice vectors do.

On the practical side, we hope that such extensions will find applications in lattice-based cryptography. One of the basic implications of Theorem \ref{thm:siegel_int}, that a ball of volume $V$ contains $V$ nonzero lattice vectors on average, is already a fundamental tool for predicting and fine-tuning the decryption process. More recent results, such as the Poissonian behavior of short lattice vectors (\cite{Sod}), have also found applications in cryptanalysis (\cite{BSW}). It is therefore natural to expect that analogous results on sublattices would be useful too, since those are what the BKZ algorithm (stands for Block Korkine-Zolotarev, originally proposed by Schnorr and Euchner \cite{SE94}), the standard decryption algorithm for lattice-based systems, operates on.

Our first result is the following generalization of Theorem \ref{thm:rogers_int}. For a lattice $L \in X_n$ and $1 \leq d < n$, write $\Gr(L, d)$ for the set of primitive rank $d$ sublattices of $L$. For an element $A \in \Gr(L,d)$, define $\det A$ as follows: choose any basis $\{v_1, \ldots, v_d\}$ of $A$, and we set $\det A = \|v_1 \wedge \ldots \wedge v_d \|$, where $\| \cdot \|$ here is the standard Euclidean norm in $\wedge^d \R^n$. This definition is independent of the basis choice. Throughout this paper, for $A \in \Gr(L, d)$ and $H \geq 0$ we write
\begin{equation*}
f_{H}(A) = \begin{cases} 1 & \mbox{if $\det A \leq H$} \\ 0 & \mbox{otherwise.} \end{cases}
\end{equation*}
We also define
\begin{equation*}
a(n,d) = \frac{1}{n}\binom{n}{d}\prod_{i=1}^{d}\frac{V(n-i+1)\zeta(i)}{V(i)\zeta(n-i+1)},
\end{equation*}
where $V(i) = \pi^{i/2}/\Gamma(1+i/2)$ is the volume of the unit ball in $\R^i$, and $\zeta(i)$ is the Riemann zeta function evaluated at the positive integer $i$, except that we pretend $\zeta(1) = 1$ for notational convenience. Then we prove the following, the first main result of the present paper.

\begin{theorem} \label{thm:main}
Suppose $1 \leq k \leq n-1$, $1 \leq d_1, \ldots, d_k \leq n-1$ with $d_1 + \ldots + d_k \leq n-1$. Then
\begin{equation*}
\int_{X_n} \underset{A_1, \ldots, A_k\ \mathrm{independent}}{\sum_{A_1 \in \Gr(L, d_1)} \ldots \sum_{A_k \in \Gr(L, d_k)}} f_{H_1}(A_1) \ldots f_{H_k}(A_k) d\mu_n(L) = \prod_{i=1}^k a(n,d_i)H_i^n.
\end{equation*}
\end{theorem}

\begin{remark}
(i) We point out that the $k=1$ case of this theorem is proved by Thunder (\cite{Thu3}) in the more general context of number fields.

(ii) If one wants a formula that counts non-primitive sublattices as well, by a standard M\"obius inversion argument (see \cite{Sch}, or Section 7.1 of \cite{Kim2}) one can show that we could simply replace all the $a(n,d)$ by
\begin{equation*}
c(n,d) = a(n,d)\prod_{i=1}^d \zeta(n-i+1).
\end{equation*}
The same applies to our other results introduced below.
\end{remark}

It is natural to ask what happens if the sublattices $A_1, \ldots, A_k$ are not independent. In fact, if $A_i \cap A_j \neq A_i, A_j, 0$ for some $i, j$, the corresponding integral diverges, as can be seen from the next statement that we prove.
\begin{theorem}\label{thm:overlap}
Let $1 \leq r < d_1, d_2$, so that $d_1 + d_2 < n + r$. Then
\begin{equation*}
\int_{X_n} \sum_{A \in \Gr(L,d_1)} \sum_{B \in \Gr(L,d_2) \atop \rk A \cap B = r} f_{H_1}(A) f_{H_2}(B) d\mu_n(L) = \infty.
\end{equation*}
\end{theorem}
Because
\begin{equation} \label{eq:l2}
\int_{X_n} \left(\sum_{A \in \Gr(L, d)} f_{H}(A) \right)^2 d\mu_n(L) = \sum_{r=0}^d \int_{X_n} \sum_{A \in \Gr(L,d)} \sum_{B \in \Gr(L,d) \atop \rk A \cap B = r} f_{H}(A) f_{H}(B) d\mu_n(L),
\end{equation}
Theorem \ref{thm:overlap} has the following consequence.
\begin{corollary}
The $L^2$-norm of $\sum_{A \in \Gr(L, d)} f_{H}(A)$ diverges.
\end{corollary}

Thus, unfortunately, it would be difficult to study $\Gr(L,d)$ via the standard methods. At least we learn that the statistics of rank $d$ sublattices is radically different from that of lattice vectors. In particular, the various Poissonian properties enjoyed by random lattice vectors (see e.g. \cite{Kim}) fail spectacularly for rank $d$ sublattices. It may be possible to tweak the left-hand side of \eqref{eq:l2}, for example by restricting to counting certain pairs of elements of $\Gr(L,d)$, to extract some useful information, but I have been unable to do so.

At the other extreme is the case in which $A_1 \subseteq \ldots \subseteq A_k$, and we would be counting \emph{flags}. For $L \in X_n$ and $d_0 = 0 < d_1 < \ldots < d_k <d_{k+1} = n$, a flag of type $\mathfrak{d} = (d_1, \ldots, d_k)$ (rational with respect to $L$) is a sequence of sublattices
\begin{equation*}
A_1 \subseteq A_2 \subseteq \ldots \subseteq A_k \subseteq L
\end{equation*}
such that $\dim A_i = d_i$. Define
\begin{equation*}
a(n, \mathfrak{d}) = a(n,d_1)\prod_{i=1}^{k-1} \frac{n-d_{i-1}}{d_{i+1}-d_{i-1}}a(n-d_i, d_{i+1} - d_i).
\end{equation*}
(This coincides with $a(\alpha)$ in Thunder (\cite{Thu2}).) Then we have the following formula for counting rational flags.

\begin{theorem} \label{thm:flags}
The $\mu_n$-average of the number of flags $A_1 \subseteq \ldots \subseteq A_k$ of type $\mathfrak{d} = (d_1, \ldots, d_k)$ such that $\det A_i \leq H_i$ for $i = 1, \ldots, k$ is equal to
\begin{equation*}
a(n, \mathfrak{d})\prod_{i=1}^k H_i^{d_{i+1} - d_{i-1}}.
\end{equation*}
\end{theorem}

Theorem \ref{thm:flags} has the following corollary.
%(perhaps surprising, in the light of \cite{FMT} or Theorem 5 of \cite{Thu2})
\begin{corollary}
The $\mu_n$-average number of flags of height $\leq H$ is equal to $\infty$.
\end{corollary}

Here, the \emph{height} of a flag $A_1 \subseteq \ldots \subseteq A_k$ is defined as the quantity
\begin{equation*}
\prod_{i=1}^k (\det A_i)^{d_{i+1}-d_{i-1}}.
\end{equation*}
(See e.g. \cite{FMT} or \cite{Thu2}.) Therefore, the average number of flags of height $\leq H$ equals
\begin{equation*}
a(n,\mathfrak{d})\int_{x_i > 0 \atop x_1 \ldots x_k \leq H} dx_1 \ldots dx_k,
\end{equation*}
but this is $\infty$ for $k \geq 2$. It may be interesting to compare this result with Theorem 5 in Thunder (\cite{Thu2}).

%%% maybe revisit the following point?
%This may be surprising in the view of Theorem 5 in Thunder's work \cite{Thu2}, which gives an asymptotic formula for the number of height $\leq H$ flags of a given lattice satisfying a certain condition. It seems as though the flags violating that condition comprise a vanishing negligible minority, yet the above corollary suggests that it is not the case. Indeed, if it were, our method for proving Theorem \ref{thm:main} would also apply to this situation, yielding that the average number of height $\leq H$ flags is equal to $(\mathrm{const}) \cdot H \log^{k-1}H$, the main term in Thunder's formula.

% continue writing from here

We end the introduction with a few words on the method of proof. It is a product of a reflection on the phenomenon in which the mean of a counting function over $X_n$ coincides with its main term for a fixed individual lattice. For instance, Theorem \ref{thm:siegel_int} implies
\begin{equation*}
\int_{X_n} \left| B(V) \cap L \backslash \{0\} \right| d\mu_n(L) = V,
\end{equation*}
which is the expected main term of an estimate of $\left| B(V) \cap L \backslash \{0\} \right|$ for a fixed $L$. Among several possible approaches to mean value formulas, we chose the one that transparently shows how this happens. A discrete analogue of Theorem 1 of Rogers (\cite{Rogers}), which bears some semblance to the Hecke equidistribution (Lemma \ref{lemma:eqdist} below), reduces the problem of averaging over $X_n$ to that of counting over a fixed lattice; and the author's recent work (\cite{Kim2}) solves this counting problem. This argument also serves to fix a recently found error in Rogers' proof of the formula named after him (\cite{Rogers}); see Section 2 below for details. For Theorems \ref{thm:overlap} and \ref{thm:flags}, we also need the standard unfolding trick for the Eisenstein series.

If so desired, an appropriate combination of these techniques allows one to handle many other integrals in which the sublattices $A_1, \ldots, A_k$ interact in different ways, e.g. something like
\begin{equation*}
\int_{X_n} \underset{A_1, A_2\ \mathrm{independent}}{\sum_{A_1 \in \Gr(L, d_1)} \sum_{A_2 \in \Gr(L, d_2)}} \sum_{B \in \Gr(A_1 \oplus A_2, r)} f_{H_1}(A_1)f_{H_2}(A_2)f_{H_3}(B) d\mu_n(L).
\end{equation*}

\subsection{Acknowledgments} 

This work was supported by NSF grant CNS-2034176. The author thanks the referee for the careful review of the original manuscript.

\section{Some background}

Fix a positive integer $n$. For a prime $p$ and $c \in \{ 1, \ldots, n \}$, let $\mathcal{M}^{(c)}_p$ be the set of $n \times n$ integral matrices $M = (m_{ij})_{1 \leq i, j \leq n}$ such that $m_{ij} = \delta_{ij}$ for $j \neq c$, and $0 \leq m_{ic} \leq p-1$ for $i < c$, $m_{cc} = p$, and $m_{ic} = 0$ for $i > c$. To illustrate, $M$ is a matrix of the form
\begin{equation*}
\begin{pmatrix} 
1 &            &     & x_1         &    &          &   \\
   & \ddots &     & \vdots     &    &           &   \\
   &            & 1  & x_{c-1} &    &           &  \\
   &            &     & p             &    &           &  \\
   &            &     &                & 1 &           & \\ 
   &            &     &                &    & \ddots & \\
   &            &     &                &    &            & 1 \end{pmatrix}.
\end{equation*}
Also let $\mathcal{M}_p = \bigcup_{c=1}^n \mathcal{M}^{(c)}_p$. It is clear that $|\mathcal{M}^{(c)}_p| = p^{c-1}$, so $|\mathcal{M}_p| = 1 + p + \ldots + p^{n-1}$. For the $n \times n$ diagonal matrix
\begin{equation*}
\alpha_p = \mathrm{diag}(1, \ldots, 1, p) = \begin{pmatrix}
1 & & & \\
   & \ddots & & \\
   &            & 1 & \\
   &            &    & p \end{pmatrix},
\end{equation*}
$\mathcal{M}_p$ serves as a set of all distinct representatives of the right cosets of $\Gamma$ in the double coset $\Gamma \alpha_p \Gamma$, where $\Gamma = \SL(n, \Z)$ --- see Chapter 3 of Shimura (\cite{Shi}). The Hecke operator on the functions on $X_n$ with respect to $\alpha_p$, which we denote by $T_p$, is defined as
\begin{align*}
T_p\rho(L) &= \frac{1}{|\Gamma \backslash \Gamma\alpha_p\Gamma|}\sum_{M \in \Gamma \backslash \Gamma\alpha_p\Gamma} \rho(p^{-1/n}ML) \\
&= \frac{1}{|\mathcal{M}_p|}\sum_{M \in \mathcal{M}_p} \rho(p^{-1/n}ML).
\end{align*}

Throughout this paper, we will also write $L$ for any representative of the coset $L \in X_n$ when it is harmless to do so, as we have done above. In other words, we write $L$ for both an element of $\SL(n,\R)$ and the lattice spanned by the row vectors of $L$, to keep the notations simple.

Lemma \ref{lemma:eqdist} below, which is a discrete analogue of Theorem 1 of Rogers (\cite{Rogers}), is crucial to the present paper. However, the author (thanks to Kevin Schmitt) recently found that Rogers's argument contains an error: in p. 256 of \cite{Rogers}, he claims $\int_{X_n} \rho(\gamma L) d\mu_n = \int_{X_n} \rho(L) d\mu_n$ for any $\gamma \in \SL(n,\R)$, saying that it follows ``from the known properties of the fundamental domain,'' but offering no justification otherwise. This equality is in fact not true, which can be seen by taking $n=2$, $\rho$ any function on $X_n$ vanishing in a neighborhood of the cusp, and
\begin{equation*}
\gamma = \begin{pmatrix} a & 0 \\ 0 & a^{-1} \end{pmatrix}
\end{equation*}
for large $a > 0$, for example. Fortunately, in our discrete context, the needed relation follows from the basic properties of the Hecke operators.

\begin{lemma} \label{lemma:hecke}
Let $\rho(L)$ be an integrable function on $X_n$. Then
\begin{equation*}
\int_{X_n} T_p\rho(L) d\mu_n(L) = \int_{X_n} \rho(L) d\mu_n(L).
\end{equation*}
\end{lemma}
\begin{proof}
The argument is very similar to the proof of Proposition 3.39 in \cite{Shi}. Write $\Gamma = \SL(n,\Z)$ as before, and let
\begin{equation*}
\widetilde{\mathcal{M}}_p^{(c)} = \{E_{cn}M : M \in \mathcal{M}_p^{(c)}\},
\end{equation*}
where $E_{cn}$ is the permutation matrix obtained by swapping the $c$-th and $n$-th rows of the $n \times n$ identity matrix and then changing the sign of the $n$-th row (so that $\det E_{cn} = 1$). Observe that $\widetilde{\mathcal{M}}_p := \bigcup_{c=1}^n \widetilde{\mathcal{M}}_p^{c}$ is also a set of all representatives of the right cosets of $\Gamma$ in $\Gamma\alpha_p\Gamma$. Moreover, $\alpha_p^{-1}M \in \Gamma$ for all $M \in \widetilde{\mathcal{M}}_p$. Therefore
\begin{equation*}
\Gamma = \Gamma \cap \alpha^{-1}_p \Gamma \alpha_p \Gamma = \bigcup_{M \in \widetilde{\mathcal{M}}_p} (\Gamma \cap \alpha_p^{-1} \Gamma \alpha_p \alpha^{-1}_p M) = \bigcup_{M \in \widetilde{\mathcal{M}}_p} (\Gamma \cap \alpha_p^{-1} \Gamma \alpha_p) \alpha^{-1}_p M.
\end{equation*}
This shows that the elements $\alpha_p^{-1}M$, $M \in \widetilde{\mathcal{M}}_p$, serve as the coset representatives of $\Gamma \cap \alpha^{-1}_p \Gamma \alpha_p$ in $\Gamma$. Hence we can reinterpret $T_p$ as
\begin{equation*}
T_p\rho(L) = \frac{1}{|\mathcal{M}_p|}\sum_{N \in (\Gamma \cap \alpha_p^{-1} \Gamma \alpha_p) \backslash \Gamma} \rho(p^{-1/n}\alpha_p NL).
\end{equation*}

For a choice of the fundamental domain $P$ with respect to $\Gamma \cap \alpha_p^{-1} \Gamma \alpha_p$, we have
\begin{align*}
&\int_{X_n} T_p\rho(L) d\mu_n(L) \\
&= \frac{1}{|\mathcal{M}_p|} \int_{X_n} \sum_{N \in (\Gamma \cap \alpha_p^{-1} \Gamma \alpha_p) \backslash \Gamma} \rho(p^{-1/n}\alpha_p NL) d\mu_n(L) \\
&= \frac{1}{|\mathcal{M}_p|} \int_{P} \rho(p^{-1/n}\alpha_p L) d\mu_n(L) \\
&= \frac{1}{|\mathcal{M}_p|} \int_{p^{-1/n}\alpha_p P} \rho(L) d\mu_n(L).
\end{align*}
However, $p^{-1/n}\alpha_p P$ is a fundamental domain with respect to $\alpha_p \Gamma \alpha^{-1}_p \cap \Gamma$. Since $|\mathcal{M}_p| = [\Gamma : \Gamma \cap \alpha_p^{-1} \Gamma \alpha_p] = [\Gamma : \alpha_p \Gamma \alpha^{-1}_p \cap \Gamma]$ (see e.g. Proposition 3.6 of \cite{Shi}), this completes the proof.

\end{proof}

\begin{lemma} \label{lemma:eqdist}
Let $\rho(L)$ be a non-negative integrable function on $X_n$. Suppose $\lim_{p \rightarrow \infty} T_p \rho(L)$ exists and have the same (finite) value $I$ for all $L \in X_n$. Then $I = \int_{X_n} \rho d\mu_n$.
\end{lemma}
\begin{proof}

For any function $F$ and $h \in \R$, write $[F]_h := \min(F, h)$. For any $h > I$, the dominated convergence theorem implies
\begin{equation*}
\int_{X_n} \left[T_p\rho\right]_h(L) d\mu_n(L) \rightarrow I
\end{equation*}
as $p \rightarrow \infty$. Also by Lemma \ref{lemma:hecke}, we have
\begin{equation*}
\int_{X_n} [\rho]_h(L) d\mu_n(L) = \int_{X_n} T_p[\rho]_h(L) d\mu_n(L) \leq \int_{X_n} [T_p\rho]_h(L) d\mu_n(L).
\end{equation*}
Taking $p \rightarrow \infty$ and then $h \rightarrow \infty$ here, by the monotone convergence theorem we obtain the upper bound
\begin{equation*}
\int_{X_n} \rho(L) d\mu_n(L) \leq I.
\end{equation*}

On the other hand, consider the integral
\begin{align*}
&\int_{X_n} T_p\rho(L) d\mu_n(L) \\
&= \int_{X_n} \frac{1}{|\mathcal{M}_p|} \sum_{M \in \mathcal{M}_p} \rho(p^{-1/n}M L) d\mu_n(L) \\
&= \frac{1}{|\mathcal{M}_p|} \sum_{M \in \mathcal{M}_p} \int_{X_n}  \rho(p^{-1/n}M L) d\mu_n(L).
\end{align*}
Combined with Fatou's lemma, this implies that
\begin{equation*}
I = \int_{X_n} \lim_{p \rightarrow \infty} T_p\rho(L) d\mu_n(L) \leq \lim_{p \rightarrow \infty} \int_{X_n} T_p\rho(L) d\mu_n(L) = \int_{X_n} \rho(L) d\mu_n(L),
\end{equation*}
again by Lemma \ref{lemma:hecke}. This completes the proof.

\end{proof}

\begin{remark}
Lemma \ref{lemma:eqdist} may remind one of the Hecke equidistribution (\cite{DM}). Of course, the latter is a much deeper statement, but Lemma \ref{lemma:eqdist} has the advantage that it applies to functions that are neither compactly supported nor bounded, which is the situation we are in.
\end{remark}

To confirm the rather strong condition for Lemma \ref{lemma:eqdist}, we need the following theorem recently established by the author (\cite{Kim2}). Here we only state the parts that we need.

\begin{theorem}[Kim \cite{Kim2}] \label{thm:fixed_L}
For a (full-rank) lattice $L \in \R^n$, define $P(L, d, H)$ to be the number of primitive rank $d < n$ sublattices of $L$ of determinant less than or equal to $H$. Also let
\begin{equation*}
b(n, d) = \max\left(\frac{1}{d}, \frac{1}{n-d}\right).
\end{equation*}
Then
\begin{equation} \label{eq:fixed_L}
P(L, d, H) = a(n, d)\frac{H^n}{(\det L)^d} + O\left(\sum_{\gamma \in \Q \atop 0 \leq \gamma \leq n - b(n,d)} b_\gamma(L)H^\gamma\right),
\end{equation}
where the implied constant depends only on $n$ and $d$, the sum on the right is finite, and every $b_\gamma$ is a reciprocal of a product of the successive minima of $L$, so that the right-hand side of \eqref{eq:fixed_L} is invariant under rescaling $L$ to $\alpha L$ and $H$ to $\alpha^d H$.

Furthermore, for a sublattice $S \subseteq L$ of rank $\leq n-d$, define $P_S(L,d,H)$ to be the number of primitive rank $d$ sublattices of $L$ of determinant $\leq H$ that intersect trivially with $S$. Then $P_S(L,d,H)$ also satisfies the estimate \eqref{eq:fixed_L}; in particular, the error term is independent of $S$. The sublattices that do intersect $S$ contribute at most $O(\sum_{\gamma \leq n-1} b_\gamma H^{\gamma})$.
\end{theorem}

For a $d \times n$ matrix $A = (a_{ij})_{d \times n}$ and $c \in \{1, \ldots, n\}$, write $A^{(c)} = (a_{ij})_{d \times (c-1)}$ for the ``first'' $d \times (c-1)$ submatrix of $A$. We also define $\det A = (\det AA^T)^{1/2}$. As with $L$, let us use the same letter $A$ to refer to the rank $d$ lattice generated by the row vectors of $A$. With this convention, the definition of $\det A$ just given for the matrix $A$ is consistent with the definition of $\det A$ for $A \in \Gr(L,d)$ in the introduction. For $A \in \Gr(L,d)$, we also sometimes write $\det_L A$ when the extra clarification might be helpful.

\begin{proposition} \label{prop:mat-det}
We continue with the notations of the preceding discussion. In addition, choose a basis $\{v_1, \ldots, v_n\}$ of $L \in X_n$, and also write $L$ for the matrix whose $i$-th row is $v_i$. Let $A$ be an integral $d \times n$ matrix, $c \in \{d+1, \ldots, n\}$, and let $\bar{L}^{(c)}$ be the $(c-1) \times n$ matrix whose $i$-th row vector $\bar{v}_i$ is the projection of $v_i$ onto $\spn\{v_c, \ldots, v_n\}^\perp$. Then, provided $\det A^{(c)} \neq 0$, $\det AL  \geq \det A^{(c)}\bar{L}^{(c)}$.
\end{proposition}
\begin{proof}
We first present the proof for the case $c = n$. Write $A = (A^{(n)}; a)$, with $a = (a_{1n}, \ldots, a_{dn})^T$. Similarly, write
\begin{equation*}
L =
\begin{pmatrix}
1 &            &    & \mu_1 \\ 
   & \ddots &    & \vdots \\
   &            & 1 & \mu_{n-1} \\
   &            &    & 1
\end{pmatrix}
\cdot
\begin{pmatrix}
 & & \\
& \bar{L}^{(n)} & \\
 & & \\
 & v_n &
\end{pmatrix}
\end{equation*}
for some $\mu_1, \ldots, \mu_{n-1} \in \R$. Write $\mu = (\mu_1, \ldots, \mu_{n-1})^T$. Then
\begin{equation*}
AL = A^{(n)}\bar{L}^{(n)} + (A^{(n)}\mu + a)v_n.
\end{equation*}
Temporarily write $\mathcal{A} = A^{(n)}\bar{L}^{(n)}$, $\mathcal{B} = (A^{(n)}\mu + a)v_n$. Then
\begin{equation*}
(AL)(AL)^T = (\mathcal{A} + \mathcal{B})(\mathcal{A}^T + \mathcal{B}^T) = \mathcal{A}\mathcal{A}^T + \mathcal{B}\mathcal{B}^T,
\end{equation*}
because $\mathcal{A}\mathcal{B}^T = \mathcal{B}\mathcal{A}^T = 0$. Also observe that
\begin{equation*}
\mathcal{B}\mathcal{B}^T = \|v_n\|^2(A^{(n)}\mu + a)(A^{(n)}\mu + a)^T.
\end{equation*}
The matrix-determinant lemma now gives
\begin{equation*}
(\det AL)^2 = (\det \mathcal{A})^2(1+ \|v_n\|^2(A^{(n)}\mu + a)^T(\mathcal{A}\mathcal{A}^T)^{-1}(A^{(n)}\mu + a)) \geq (\det \mathcal{A})^2,
\end{equation*}
which yields the desired conclusion.

To prove the $c=n-1$ case, for example, observe that we can write
\begin{equation*}
\bar{L}^{(n)} = \begin{pmatrix}
1 &            &    & \mu'_1 \\ 
   & \ddots &    & \vdots \\
   &            & 1 & \mu'_{n-2} \\
   &            &    & 1
\end{pmatrix}
\cdot
\begin{pmatrix}
 & & \\
& \bar{L}^{(n-1)} & \\
 & & \\
 & \bar{v}_{n-1} &
\end{pmatrix}
\end{equation*}
for some $\mu'_1, \ldots, \mu'_{n-2} \in \R$, where $\bar{v}_{n-1}$ is the last row of $\bar{L}^{(n)}$. By repeating the argument above, we obtain $\det A^{(n)}\bar{L}^{(n)} \geq \det A^{(n-1)}\bar{L}^{(n-1)}$. The remaining cases follow by further repetitions.
\end{proof}

\section{Proof of Theorem \ref{thm:main}}

Recall that we wish to evaluate the integral
\begin{equation} \label{eq:toy2}
\int_{X_n} \underset{A_1, \ldots, A_k\ \mathrm{independent}}{\sum_{A_1 \in \Gr(L, d_1)} \ldots \sum_{A_k \in \Gr(L, d_k)}} f_{H_1}(A_1) \ldots f_{H_k}(A_k) d\mu_n(L),
\end{equation}
where we have $d_1 + \ldots + d_k := d < n$. The plan is to instead estimate the sum
\begin{equation} \label{eq:toy2sum}
\frac{1}{|\mathcal{M}_p|} \sum_{M \in \mathcal{M}_p}\underset{A_1, \ldots, A_k\ \mathrm{independent}}{\sum_{A_1 \in \Gr(\Z^n, d_1)} \ldots \sum_{A_k \in \Gr(\Z^n, d_k)}} f_{p^{d_1/n}H_1}(A_1ML) \ldots f_{p^{d_k/n}H_k}(A_kML)
\end{equation}
for a fixed $L \in X_n$ in the $p$ limit, and then use Lemma \ref{lemma:eqdist} to prove Theorem \ref{thm:main}. More precisely,  we rewrite \eqref{eq:toy2sum} as
\begin{equation*}
\frac{1}{|\mathcal{M}_p|} \sum_{c=1}^{n} \sum_{M \in \mathcal{M}^{(c)}_p}\underset{A_1, \ldots, A_k\ \mathrm{independent}}{\sum_{A_1 \in \Gr(\Z^n, d_1)} \ldots \sum_{A_k \in \Gr(\Z^n, d_k)}} f_{p^{d_1/n}H_1}(A_1ML) \ldots f_{p^{d_k/n}H_k}(A_kML),
\end{equation*}
and study the inner sum for each $c$, to show that this approaches $\prod_{i=1}^k a(n,d_i)H_i^n$ as $p \rightarrow \infty$. This is independent of $L$, and hence Lemma \ref{lemma:eqdist} applies.

Throughout this section, we fix a representative of $L$ in $\SL(n,\R)$, also denoted by $L$. It will also be helpful for us to identify each $A_l \in \Gr(\Z^n,d_l)$ with a choice of its matrix representative, e.g. its Hermite normal form (HNF), for explicit computations. Let $c \in \{1, \ldots, n\}$, and
\begin{equation*}
M = \begin{pmatrix} 
1 &            &     & x_1         &    &          &   \\
   & \ddots &     & \vdots     &    &           &   \\
   &            & 1  & x_{c-1} &    &           &  \\
   &            &     & p             &    &           &  \\
   &            &     &                & 1 &           & \\ 
   &            &     &                &    & \ddots & \\
   &            &     &                &    &            & 1 \end{pmatrix} \in \mathcal{M}^{(c)}_p.
\end{equation*}
If we write $A_l = (a^l_{ij})_{d_l \times n}$ as a matrix, then for each $l$
\begin{equation*}
A_lML =
\begin{pmatrix}
a^l_{11}    & a^l_{12}    & \dots & \sum_{j=1}^{c-1} a^l_{1j}x_j + pa^l_{1c}      & \dots & a^l_{1n} \\
  \vdots     &  \vdots      &          & \vdots                                                           &          & \vdots \\
a^l_{d_l1} & a^l_{d_l2} & \dots & \sum_{j=1}^{c-1} a^l_{d_lj}x_j + pa^l_{d_lc} & \dots & a^l_{d_ln}
\end{pmatrix}
L.
\end{equation*}
Denote by $A^{(c)}_l = (a^l_{ij})_{d_l \times (c-1)}$ the first $d_l \times (c-1)$ submatrix of $A_l$, and write
\begin{equation*}
A = (a_{ij})_{d \times n} = \begin{pmatrix} A_1 \\ \vdots \\ A_k \end{pmatrix},
A^{(c)} = (a_{ij})_{d \times (c-1)} = \begin{pmatrix} A^{(c)}_1 \\ \vdots \\ A^{(c)}_k \end{pmatrix},
\end{equation*}
which are $d \times n$ and $d \times (c-1)$ matrices, respectively.

\subsection{The main term}

The main term of \eqref{eq:toy2sum} comes from the terms with $c=n$, namely
\begin{equation}\label{eq:c=n}
\frac{1}{|\mathcal M_{p}|}\sum_{M \in \mathcal{M}^{(n)}_p}\underset{A_1, \ldots, A_k\ \mathrm{independent}}{\sum_{A_1 \in \Gr(\Z^n, d_1)} \ldots \sum_{A_k \in \Gr(\Z^n, d_k)}} f_{p^{d_1/n}H_1}(A_1ML) \ldots f_{p^{d_k/n}H_k}(A_kML).
\end{equation}

Consider the map $A^{(n)}: \F_p^{n-1} \rightarrow \F_p^d$ induced by the matrix $A^{(n)}$ as above. We claim that the contribution to the above sum from the $A_l$'s for which $A^{(n)}$ is not surjective is negligible. There are two cases:
\begin{enumerate}[(i)]
\item $\det A^{(n)} \geq p$ over $\Q$. Then
\begin{equation*}
p \leq \det A^{(n)} \leq \prod_{l=1}^k \det A^{(n)}_l,
\end{equation*}
so there exists an $l$ such that $\det A^{(n)}_l \geq p^{d_l/n + 1/nk}$, and hence $\det A^{(n)}_l\bar{L} \gg_L p^{d_l/n + 1/nk}$, where $\bar{L}$ here is $\bar{L}^{(n)}$ in the statement of Proposition \ref{prop:mat-det}. Proposition \ref{prop:mat-det} implies $\det A_lML \gg_L p^{d_l/n + 1/nk}$. For $p$ sufficiently large, this is greater than $p^{d_l/n}H_l$, and so does not contribute to the sum.

\item $\det A^{(n)} = 0$ over $\Q$. Row-reduce $A$ so that the last row equals $(0, \ldots, 0, C)$ with $C \neq 0$ --- possible by assumption $\rk A = d$ --- which shows that $\det AM \geq p$. This again implies $\det A_lML \gg_L p^{d_l/n + 1/nk}$ for some $l$.
\end{enumerate}

Now if $A^{(n)}$ did induce a surjective map onto $\F_p^d$, then the vectors
\begin{equation*}
\left( \sum_{j=1}^{n-1} {a_{1j}x_j}, \ldots, \sum_{j=1}^{n-1} {a_{dj}x_j} \right)
\end{equation*}
are equidistributed mod $p$. Therefore computing \eqref{eq:c=n} is equivalent to computing
\begin{equation*}
\frac{p^{n-1}}{p^d|\mathcal M_{p}|}\underset{\rk_p A^{(n)} = d}{\sum_{A_1 \in \Gr(\Z^n,d_1)} \ldots \sum_{A_k \in \Gr(\Z^n,d_k)}} f_{p^{d_1/n}H_1}(A_1L) \ldots f_{p^{d_k/n}H_k}(A_kL)
\end{equation*}
(here $\rk_p$ means the rank modulo $p$). This is equal to
\begin{equation} \label{eq:ribbit}
\frac{p^{n-1}}{p^d|\mathcal M_{p}|}\underset{\rk A^{(n)} = d}{\sum_{A_1 \in \Gr(\Z^n,d_1)} \ldots \sum_{A_k \in \Gr(\Z^n,d_k)}} f_{p^{d_1/n}H_1}(A_1L) \ldots f_{p^{d_k/n}H_k}(A_kL)
\end{equation}
because $\rk_p A^{(n)} = d \Leftrightarrow \rk A^{(n)} = d$ and $p \nmid \det A^{(n)}$, and we already showed that if $p \mid \det A^{(n)}$ the corresponding sets of $A_l$'s do not contribute to the sum.

The summation in \eqref{eq:ribbit} requires that $A_k$ is independent of $A_1, \ldots, A_{k-1}$ and $(0, \ldots, 0, 1)$, since otherwise, $A_1 \oplus \ldots \oplus A_k$ contains a nonzero multiple of $(0, \ldots, 0, 1)$, which implies $\rk A^{(n)} < d$. Thus, by applying Theorem \ref{thm:fixed_L} with $S = A_1 \oplus \ldots \oplus A_{k-1} \oplus \spn_\Z\{(0, \ldots, 0, 1)\}$, we can rewrite \eqref{eq:ribbit} as
\begin{align*}
&\frac{p^{n-1}}{p^d|\mathcal M_{p}|}\left(a(n,d_k)H_k^np^{d_k} + o_L(H_k^np^{d_k})\right) \cdot \\
&\underset{\rk A^{(n)} = d - d_k}{\sum_{A_1 \in \Gr(\Z^n,d_1)} \ldots \sum_{A_{k-1} \in \Gr(\Z^n,d_{k-1})}} f_{p^{d_1/n}H_1}(A_1L) \ldots f_{p^{d_{k-1}/n}H_{k-1}}(A_{k-1}L),
\end{align*}
where $A^{(n)}$ here now means
\begin{equation*}
A^{(n)} = \begin{pmatrix} A^{(n)}_1 \\ \vdots \\ A^{(n)}_{k-1} \end{pmatrix}.
\end{equation*}

Repeating the same argument with other $A_i$'s, we find that \eqref{eq:ribbit} equals
\begin{equation*}
\frac{p^{n-1}}{|\mathcal M_{p}|}\left(\prod_{i=1}^ka(n,d_i)H^n_i + o_{L,H_1, \ldots, H_k}(1)\right).
\end{equation*}
Recalling $|\mathcal M_{p}| = \sum_{i=0}^{n-1} p^i$, and taking $p \rightarrow \infty$, this gives the intended main term $\prod_{i=1}^k a(n,d_i)H^n_i$ for \eqref{eq:toy2sum}.

\subsection{Error terms, part 1}

In the rest of this section, we show that, for $c \in \{1, \ldots, n-1\}$,
\begin{equation}\label{eq:c<n}
\frac{1}{|\mathcal M_{p}|}\sum_{M \in \mathcal{M}^{(c)}_p}\underset{A_1, \ldots, A_k\ \mathrm{independent}}{\sum_{A_1 \in \Gr(\Z^n, d_1)} \ldots \sum_{A_k \in \Gr(\Z^n, d_k)}} f_{p^{d_1/n}H_1}(A_1ML) \ldots f_{p^{d_k/n}H_k}(A_kML)
\end{equation}
vanishes as $p \rightarrow \infty$. This will complete the proof of Theorem \ref{thm:main}.

We first assume $c > d$, and consider the contributions from those $A_1, \ldots, A_k$ such that $\rk_p A^{(c)} = d$. By a similar argument to the $c=n$ case, the surjection of the linear map $A^{(c)} : \mathbb{F}^{c-1}_p \rightarrow \mathbb{F}^d_p$ implies that their contributions amount to
\begin{equation*}
\frac{p^{c-1}}{p^d|\mathcal M_{p}|}\underset{\rk_p A^{(c)} = d}{\sum_{A_1 \in \Gr(\Z^n,d_1)} \ldots \sum_{A_k \in \Gr(\Z^n,d_k)}} f_{p^{d_1/n}H_1}(A_1L) \ldots f_{p^{d_k/n}H_k}(A_kL).
\end{equation*}
We simply drop the rank condition and bound this by
\begin{equation*}
\frac{p^{c-1}}{|\mathcal M_{p}|}\left(\prod_{i=1}^k a(n,d_i)H_i^n + o_{L, H_1, \ldots, H_k}(1)\right),
\end{equation*}
which clearly vanishes as $p \rightarrow \infty$.

Continue with the assumption $c > d$, but this time suppose $\rk_p A^{(c)} < d$. If $\det A^{(c)} \geq p$ (over $\Q$), then we can argue exactly as in (i) in Section 3.1 above and show it does not contribute to \eqref{eq:c<n}. The case $\det A^{(c)} = 0$ will be handled below. %%% continue from here..
% 1. we can assume rk A^{(c)} = c'. if rk A^{(c)} > c', that implies det A >= p
% 2. if rk A^{(c+1)} = c' + 1, then easy
% 3. if rk A^{(c+1)} = c', that's the most annoying part --- because we cannot entirely discard its contribution.

\subsection{Error terms, part 2}

We now assume that either $c > d$ and $\det A^{(c)} = 0$, or $c \leq d$, in which case $\det A^{(c)} = 0$ necessarily. Write $\rk_p A^{(c)} = c' < \min(c,d)$. We claim that we may assume $\rk A^{(c)} = c'$ as well. If not, then $\rk A^{(c)} > \rk_p A^{(c)}$, and thus the HNF of $A^{(c)}$ has a leading coefficient (also called a pivot) that is a nonzero multiple of $p$. But this implies that $\det A \geq p$ by the Cauchy-Binet formula, and we can again argue as in (i) in Section 3.1 to show that this $A$ contributes zero to \eqref{eq:c<n}.

Suppose in addition that $\rk A^{(c+1)} = c' + 1$. Then the HNF of $A$ has a pivot in column $c$, and it follows that the HNF of $AM$ has a pivot in column $c$ that is a multiple of $p$. This implies $\det AM \geq p$, and again we argue as in (i) in Section 3.1.

Summarizing our argument so far, it remains to consider the case in which $\det A^{(c)} = 0$, and $\rk_p A^{(c)} = \rk A^{(c)} = \rk A^{(c+1)} = c' < \min(c,d)$. For integers $1 \leq r_1 < r_2 < \ldots < r_{c'} \leq d$, let $r = (r_1, \ldots, r_{c'})$, and for a matrix $B$ with $d$ rows, denote by $B|_r$ the matrix with $c'$ rows whose $i$-th row is the $r_i$-th row of $B$. For each $r$, let us restrict \eqref{eq:c<n} to those $A$ for which the rows of $A^{(c)}|_r$ are linearly independent. Thus, we are considering the following restriction of \eqref{eq:c<n}:

\begin{equation*}
\frac{1}{|\mathcal M_{p}|}\sum_{M \in \mathcal{M}^{(c)}_p}\underset{\rk A = d,\, \rk A^{(c)}|_r = \rk_p A^{(c)} = \rk A^{(c)} = \rk A^{(c+1)} = c'}{\sum_{A_1 \in \Gr(\Z^n, d_1)} \ldots \sum_{A_k \in \Gr(\Z^n, d_k)}} f_{p^{d_1/n}H_1}(A_1ML) \ldots f_{p^{d_k/n}H_k}(A_kML).
\end{equation*}
For each $A$ appearing in this sum, there exists a rational $d \times c'$ matrix $R$ such that $A^{(c)} = RA^{(c)}|_r$. Due to the rank condition $\rk A^{(c)} = \rk A^{(c+1)}$, we also must have $A^{(c+1)} = RA^{(c+1)}|_r$. In other words, $R$ and $A^{(c+1)}|_r$ determine $A^{(c+1)}$ under our current assumptions. With this understanding, we can rewrite the above sum as
\begin{align*}
&\frac{1}{|\mathcal M_{p}|} \sum_R \sum_{M \in \mathcal{M}^{(c)}_p} \\
&\underset{\rk A = d, A^{(c)} = RA^{(c)}|_r, \atop \rk A^{(c)}|_r = \rk_p A^{(c)} = \rk A^{(c)} = \rk A^{(c+1)} = c'}{\sum_{A_1 \in \Gr(\Z^n, d_1)} \ldots \sum_{A_k \in \Gr(\Z^n, d_k)}} f_{p^{d_1/n}H_1}(A_1ML) \ldots f_{p^{d_k/n}H_k}(A_kML),
\end{align*}
where the sum over $R$ is over all $d \times c'$ matrices such that the inner sum is nontrivial. Similarly to the argument in Section 3.1, for each $A$ appearing in the sum, as $M$ ranges over $\mathcal{M}^{(c)}_p$, the $c$-th column of $A|_rM$
\begin{equation*}
\left( \sum_{j=1}^{c-1} a_{r_1j} x_j + pa_{r_1c}, \ldots, \sum_{j=1}^{c-1} a_{r_{c'}j} x_j + pa_{r_{c'}c} \right)^T
\end{equation*}
becomes equidistributed mod $p$. Also, multiplying by $M$ from the right keeps all the rank conditions invariant. Thus the above sum becomes
\begin{align}
\frac{p^{c-1}}{p^{c'}|\mathcal{M}_p|} \sum_R \underset{\rk A = d, A^{(c)} = RA^{(c)}|_r, \atop \rk A^{(c)}|_r = \rk_p A^{(c)} = \rk A^{(c)} = \rk A^{(c+1)} = c'}{\sum_{A_1 \in \Gr(\Z^n, d_1)} \ldots \sum_{A_k \in \Gr(\Z^n, d_k)}} f_{p^{d_1/n}H_1}(A_1L) \ldots f_{p^{d_k/n}H_k}(A_kL) \notag \\
= \frac{p^{c-1}}{p^{c'}|\mathcal{M}_p|}\underset{\rk A = d,\, \rk A^{(c)}|_r = \rk A^{(c)} = \rk A^{(c+1)} = c'
%, \atop A^{(c+1)}|_{\tilde{r}}\mathrm{\ dependent\ on\ } A^{(c+1)}|_{r}
}{\sum_{A_1 \in \Gr(\Z^n,d_1)} \ldots \sum_{A_k \in \Gr(\Z^n,d_k)}} f_{p^{d_1/n}H_1}(A_1L) \ldots f_{p^{d_k/n}H_k}(A_kL). \label{eq:abcde}
\end{align}

It remains to estimate this sum \eqref{eq:abcde}. By dropping the rank conditions and applying Theorem \ref{thm:fixed_L}, \eqref{eq:abcde} is at most
\begin{equation*}
\frac{p^{c-c'+d-1}}{|\mathcal{M}_p|}\left(\prod_{i=1}^k a(n,d_i)H_i^{n} + o_{L,H_1, \ldots, H_k}(1)\right),
\end{equation*}
which approaches $0$ as $p \rightarrow \infty$ provided $n-c > d-c'$ (note that $n-c \geq d-c'$ always). If $n-c = d-c'$, note that the HNF of $A$ is of the form
\begin{equation*}
\begin{pmatrix} P & * \\ 0 & Q \end{pmatrix},
\end{equation*}
where $P$ is a $c' \times c$ matrix, and $Q$ is a $(d-c') \times (n-c)$ matrix, where $P$ and $Q$ are themselves HNFs and must be of full rank. Now since $n-c=d-c'$ by assumption, $Q$ is a square matrix, and therefore $Q$ must be an upper triangular matrix with no nonzero diagonal entries. This implies that $A_1 \oplus \ldots \oplus A_k$ intersects nontrivially with $\spn_\Z\{(0, \ldots, 0, 1)\}$. Therefore, one of the $A_l$'s intersects nontrivially with the lattice $S_l := A_1 \oplus \ldots \oplus \hat{A}_l \oplus \ldots \oplus A_{k} \oplus \spn_\Z\{(0, \ldots, 0, 1)\}$ --- where $\hat{A}_l$ here indicates that $A_l$ does \emph{not} appear in the sum --- of dimension at most $d-d_l+1$. But Theorem \ref{thm:fixed_L} implies 
\begin{equation*}
\sum_{A_l \in \Gr(\Z^n, d_l) \atop A_l \cap S_l \neq 0} f_{p^{d_l/n}H_l}(A_lL) = o_{L, H_1, \ldots, H_k}(p^{d_l}),
\end{equation*}
which implies that \eqref{eq:abcde} is bounded by
\begin{equation*}
o_{L,H_1, \ldots, H_k}\left(\frac{p^{c-c'+d-1}}{|\mathcal{M}_p|}\right) = o_{L,H_1, \ldots, H_k}(1),
\end{equation*}
as desired.

\section{The case of partially overlapping sublattices}

In this section we prove Theorem \ref{thm:overlap}. Recall that we are considering the integral
\begin{equation} \label{eq:toy4}
\int_{X_n} \sum_{A \in \Gr(L,d_1)} \sum_{B \in \Gr(L,d_2) \atop \rk A \cap B = r} f_{H_1}(A) f_{H_2}(B) d\mu_n(L)
\end{equation}
for $1 \leq r < d_1, d_2$, so that $d_1 + d_2 < n + r$.

Let $L \in X_n$. For a primitive sublattice $C \subseteq L$ of rank $r < n$, we identify the quotient lattice $L/C$ with the projection of $L$ onto the $n-r$-dimensional subspace of $\R^n$ orthogonal to $C$, and assign the metric induced by this projection from the metric on $L$. If $A \in \Gr(L,d)$ satisfies $C \subseteq A$, then it is easy to see that $A/C \in \Gr(L/C, d-r)$, and that it satisfies 
\begin{equation*}
{\det}_L(A) = {\det}_L(C) {\det}_{L/C}(A/C).
\end{equation*}
Using this relation, we rewrite the inner sum of \eqref{eq:toy4} as
\begin{equation*}
\sum_{C \in \Gr(L, r)} \sum_{\bar{A} \in \Gr(L/C, d_1-r) \atop { \bar{B} \in \Gr(L/C, d_2-r) \atop \mathrm{indep}}} f_{H_1/\det_L C}(\bar{A}) f_{H_2/\det_L C}(\bar{B}).
\end{equation*}

We will interpret this expression as a pseudo-Eisenstein series, i.e. a function on $X_n$ of form $\sum_{\gamma \in P \cap \Gamma \backslash \Gamma} f(\gamma L)$ where $\Gamma = \SL(n,\Z)$ and $P$ is a parabolic subgroup of $\SL(n,\R)$, and then use the unfolding trick. Fix a representative of $L \in X_n$ in $\SL(n,\R)$, again denoted by $L$. In this context, a choice of $C \in \Gr(L, r)$ corresponds to two choices of $\gamma \in P(n-r,r,\Z) \backslash \SL(n,\Z)$, where
\begin{equation*}
P(a,b,F) = \left\{
\begin{pmatrix}
G_a & *  \\
& G_b
\end{pmatrix}
: G_a \in \SL(a, F), G_b \in \SL(b, F)
\right\}
\end{equation*}
for any ring $F$. $C$ and $\gamma$ are related by $C = \mbox{(sublattice generated by last $r$ rows of $\gamma L$)}$, and this correspondence is one-to-two due to the two possible orientations for the last $r$ rows of $\gamma L$.

Next, fixing a representative of $\gamma$, we can uniquely decompose $\gamma L$ in the form
\begin{equation} \label{eq:levi_decomp}
\begin{pmatrix} G_{n-r} & U \\  & G_r \end{pmatrix} \begin{pmatrix} \alpha^{-\frac{1}{n-r}}I_{n-r} &  \\  & \alpha^{\frac{1}{r}}I_r \end{pmatrix} K',
\end{equation}
for some $G_{n-r} \in \SL(n-r,\R)$, $G_{r} \in \SL(r,\R)$, $U \in \mathrm{Mat}_{(n-r) \times r}(\R)$, and $K'$ an element of a fundamental domain of $(\SO(n-r,\R) \times \SO(r,\R)) \backslash \SO(n,\R)$, which we fix in advance. In this context, a choice of an independent pair $\bar{A} \in \Gr(L/C, d_1-r)$ and $\bar{B} \in \Gr(L/C, d_2-r)$  corresponds to four choices --- again due to the four possibilities for the orientations --- of $\delta \in P'(n-d_1-d_2+r, d_2-r, d_1-r,\Z) \backslash \SL(n-r, \Z)$, where
\begin{equation*}
P'(a,b,c, F) = \left\{
\begin{pmatrix}
G_a & * & * \\
 & G_b & 0 \\
 &   &  G_c
\end{pmatrix}
: G_a \in \SL(a, F), G_b \in \SL(b, F), G_c \in \SL(c, F)
\right\}
\end{equation*}
for any ring $F$. Indeed, $\bar{A}$ is the sublattice generated by the last $d_1-r$ rows of
\begin{equation}\label{eq:temtem}
\delta \cdot G_{n-r} \cdot \alpha^{-\frac{1}{n-r}}I_{n-r} \cdot \left(\mbox{first $n-r$ rows of $K'$} \right);
\end{equation}
this expression is independent of $U$ because $\bar A$ is orthogonal to $C$, and equivalently also to the last $r$ rows of $K'$. Similarly, $\bar{B}$ is the sublattice generated by the next last $d_2-r$ rows of \eqref{eq:temtem}. But since we are only interested in the determinants of $\bar{A}$ and $\bar{B}$, in what follows we can regard them as the sublattices generated by the corresponding rows of $\delta \cdot \alpha^{-\frac{1}{n-r}}G_{n-r}$.

The considerations so far allows us to rewrite \eqref{eq:toy4} as
\begin{equation*}
\frac{1}{8}\int_{P(n-r,r,\Z) \backslash \SL(n,\R)} \sum_{\delta} f_{H_1/\alpha}(\bar{A}) f_{H_2/\alpha}(\bar{B}) d\mu_n,
\end{equation*}
where $\delta$ is summed over $P'(n-d_1-d_2+r, d_2-r, d_1-r,\Z) \backslash \SL(n-r, \Z)$.

Lemma \ref{lemma:measure} below, which gives a decomposition of $d\mu_n$ compatible with \eqref{eq:levi_decomp}, implies that this equals
\begin{equation*}
\mathrm{(const)}\int_0^\infty \int_{X_{n-r}} \sum_{\bar{A} \in \Gr(\alpha^{-\frac{1}{n-r}}G_{n-r},d_1-r) \atop {\bar{B} \in \Gr(\alpha^{-\frac{1}{n-r}}G_{n-r},d_2-r) \atop \mathrm{indep}}} f_{H_1/\alpha}(\bar{A}) f_{H_2/\alpha}(\bar{B}) d\mu_{n-r}(G_{n-r}) \cdot \alpha^{n-1} d\alpha,
\end{equation*}
which, by Theorem \ref{thm:main} (notice that $\det\alpha^{-\frac{1}{n-r}}G_{n-r} = \alpha^{-1}$, so we normalize accordingly), equals
\begin{equation*}
\mathrm{(const)}\int_0^\infty H_1^{n-r}H_2^{n-r}\alpha^{-n+d_1+d_2-1}d\alpha,
\end{equation*}
which is divergent, proving Theorem \ref{thm:overlap}. One way to understand this phenomenon is that, if $L \in X_n$ is heavily skewed i.e. its successive minima have a huge gap, there may exist too many possibilities for $C = A \cap B$. Indeed, if we additionally required that $\alpha = \det C \geq 1$, say, then we would have instead obtained
\begin{equation*}
\mathrm{(const)}\int_1^\infty H_1^{n-r}H_2^{n-r}\alpha^{-n+d_1+d_2-1}d\alpha,
\end{equation*}
which converges, at least when $d_1+d_2 < n$.

Before we prove the needed lemma, we fix our notations related to $d\mu_n$. Recall the standard fact (see e.g. \cite{KV}) that, with respect to the $NAK$ decomposition of $\SL(n,\R)$, we can write $d\mu_n$ as
\begin{equation*}
d\mu_n = \tau(n)dN \cdot dA \cdot dK,
\end{equation*}
where $\tau(n)$ is some constant, $dN = \prod_{i < j} dn_{ij}$, $dA = \prod_i \alpha_i^{-i(n-i)} d\alpha_i / \alpha_i$ upon writing $A = \mathrm{diag}(a_1, \ldots, a_n)$ and  $\alpha_i = a_i/a_{i+1}$, and $dK$ is the Haar measure on $\SO(n,\R)$ so that
\begin{equation*}
\int_{\SO(n,\R)} dK = \prod_{i=2}^n iV(i).
\end{equation*}
To make $\int_{X_n} d\mu_n = 1$, we set
\begin{equation*}
\tau(n) = \frac{1}{n}\prod_{i=2}^n \zeta^{-1}(i).
\end{equation*}

\begin{lemma}\label{lemma:measure}

With respect to the decomposition \eqref{eq:levi_decomp} of $\SL(n,\R)$, we have
\begin{equation*}
d\mu_n = \frac{n}{r(n-r)} \cdot \frac{\tau(n)}{\tau(r)\tau(n-r)}dU d\mu_r(G_r) d\mu_{n-r}(G_{n-r}) \alpha^{n-1} d\alpha dK',
\end{equation*}
where $dU = \prod_{1 \leq i \leq n-r \atop 1 \leq j \leq r} du_{ij}$ on writing $U = (u_{ij})_{1 \leq i \leq n-r \atop 1 \leq j \leq r}$, and $dK'$ is the natural measure on $(\SO(n-r,\R) \times \SO(r,\R)) \backslash \SO(n,\R)$ descended from the measure $dK$ on $\SO(n,\R)$.

\end{lemma}

\begin{proof}
The only nontrivial part of the proof consists of comparing the diagonal parts, or the ``$A$ parts'', of the measures $d\mu_{n}, d\mu_{n-r}, d\mu_r$. We can decompose 
\begin{equation*}
\begin{pmatrix}
a_1 & & & & \\
 & & & & \\
 & & \ddots & & \\
 & & & & \\
 & & & & a_n
\end{pmatrix}
=
\begin{pmatrix}
b'_1 & & & & & \\
 & \ddots & & & & & \\
 & & b'_{n-r} & & & \\
 & & & b''_1 & & \\
 & & & & \ddots & \\
 & & & & & b''_r
\end{pmatrix}
\begin{pmatrix}
\alpha^{\frac{-1}{n-r}} & & & & & \\
 & \ddots & & & & & \\
 & & \alpha^{\frac{-1}{n-r}}  & & & \\
 & & & \alpha^{\frac{1}{r}}  & & \\
 & & & & \ddots & \\
 & & & & & \alpha^{\frac{1}{r}}
\end{pmatrix},
\end{equation*}
where there is the relation $\prod a_i = \prod b'_i = \prod b''_i = 1$ among the entries. Write $\alpha_i = a_i/a_{i+1}, \beta'_i = b'_i/b'_{i+1}, \beta''_i = b''_i/b''_{i+1}$. Then the measures on the ``A parts'' of the groups $G_n, G_{n-r}, G_r$ are, respectively,
\begin{equation*}
dA := \prod_i \alpha_i^{-i(n-i)}\frac{d\alpha_i}{\alpha_i}, dB' = \prod_i \beta_i'^{-i(n-r-i)}\frac{d\beta'_i}{\beta'_i}, dB'' = \prod_i \beta''^{-i(r-i)}\frac{d\beta''_i}{\beta''_i}.
\end{equation*}

It remains to perform the change of coordinates from the $\alpha_i$-coordinates to the $\beta'_i, \beta''_i, \alpha$-coordinates. We have $\alpha_i = \beta'_i$ for $1 \leq i \leq n-r-1$ and $\alpha_{n-r+i} = \beta''_i$ for $1 \leq i \leq r-1$, and the single nontrivial relation
\begin{equation*}
\alpha_{n-r} = \frac{b'_{n-r}}{b''_1}\alpha^{-\frac{n}{r(n-r)}}.
\end{equation*}

At this point, we can compute and find that
\begin{equation} \label{eq:subsub}
dA = dB'dB'' \cdot \prod_{i=1}^{n-r-1} \beta'^{-ri}_i \prod_{i=1}^{r-1} \beta_{i}''^{-(n-r)(r-i)} \cdot \alpha_{n-r}^{-r(n-r)}\frac{d\alpha_{n-r}}{\alpha_{n-r}}.
\end{equation}

On the other hand, from the shape of the Jacobian matrix
\begin{equation*}
\begin{pmatrix}
- & \frac{\partial\alpha_i}{\partial\beta'_1} & - \\
  & \vdots & \\
- & \frac{\partial\alpha_i}{\partial\alpha} & - \\
  & \frac{\partial\alpha_i}{\partial\beta''_1} &\\
- & \vdots & - 
\end{pmatrix}
= \begin{pmatrix}
1 &            &    & *                                                        &    &            &    \\
   & \ddots &    & \vdots                                                &    &            &    \\
   &            & 1 & *                                                        &    &            &   \\
   &            &    & \frac{\partial\alpha_{n-r}}{\partial\alpha}  &    &           &    \\
   &            &    & *                                                        & 1 &            &    \\
   &            &    & \vdots                                                &    & \ddots &    \\
   &            &    & *                                                        &    &            & 1 
\end{pmatrix}
\end{equation*}
and the fact that
\begin{equation*}
\frac{\partial\alpha_{n-r}}{\partial\alpha} = -\frac{n}{r(n-r)}\alpha_{n-r}\alpha^{-1},
\end{equation*}
we have
\begin{equation*}
-\frac{d\alpha_{n-r}}{\alpha_{n-r}} = \frac{n}{r(n-r)}\frac{d\alpha}{\alpha} + \mbox{(terms in $d\beta'$ and $d\beta''$)},
\end{equation*}
and the $d\beta'$ and $d\beta''$ parts here do not affect the outcome of the computation. Thus we can pretend that we have
\begin{align*}
&\alpha_{n-r}^{-r(n-r)}\frac{d\alpha_{n-r}}{\alpha_{n-r}} \\
&= \frac{-n}{r(n-r)}\left(\frac{b'_{n-r}}{b''_1}\alpha^{-\frac{n}{r(n-r)}}\right)^{-r(n-r)}\frac{d\alpha}{\alpha} \\
&= \frac{-n}{r(n-r)} \frac{\prod_{i=1}^{n-r-1} \beta'^{ri}_i}{\prod_{i=1}^{r-1} \beta_{i}''^{-(n-r)(r-i)}}\alpha^{n} \frac{d\alpha}{\alpha}.
\end{align*}

Substituting this into \eqref{eq:subsub}, we obtain
\begin{equation*}
dA = \frac{-n}{r(n-r)}dB'dB''\alpha^n\frac{d\alpha}{\alpha},
\end{equation*}
which completes the proof. By reorienting $\alpha$ (which moves in the opposite direction to $\alpha_{n-r}$) we can eliminate the negative sign.

\end{proof}

\section{Average number of flags}

The same technique as in the previous section can be applied to compute the $\mu_n$-average number of flags bounded by certain constraints, even though such a formula for a fixed lattice is not known and is probably difficult to find. In this section, we compute the average number of flags such that $\det A_i \leq H_i$ for $i = 1, \ldots, k$, i.e. the quantity
\begin{equation*}
\int_{X_n} \sum_{A_1 \subseteq \ldots \subseteq A_k \subseteq L \atop \dim A_i = d_i} \prod_{i=1}^k f_{H_i}(A_i) d\mu_n,
\end{equation*}
or equivalently
\begin{equation*}
\int_{X_n} \sum_{A_1 \in \Gr(L, d_1)} \sum_{\bar{A}_2 \in \Gr(L/A_1, d_2-d_1)} \ldots \sum_{\bar{A}_k \in \Gr(L/A_{k-1}, d_k-d_{k-1})} f_{H_1}(A_1)f_{H_2/\det A_1}(A_2) \ldots d\mu_n,
\end{equation*}
thereby proving Theorem \ref{thm:flags}.

First consider the case $k=2$, in which we are computing
\begin{equation*}
\int_{X_n} \sum_{A_1 \in \Gr(L, d_1)} \sum_{\bar{A}_2 \in \Gr(L/A_1, d_2-d_1)}  f_{H_1}(A_1)f_{H_2/\det A_1}(A_2) d\mu_n.
\end{equation*}
By the same argument as in the previous section, and Lemma \ref{lemma:measure}, this equals
\begin{align*}
&\frac{n}{d_1(n-d_1)} \cdot \frac{\tau(n)}{\tau(d_1)\tau(n-d_1)} \cdot \frac{1}{2}\mathrm{vol}\left(\frac{\SO(n,\R)}{\SO(n-d_1,\R) \times \SO(d_1,\R)}\right) \cdot \\
&\int \int_{X_{d_1}}d\mu_{d_1} \int_{X_{n-d_1}} \sum_{\bar{A}_2} f_{H_2/\alpha}(\bar{A}_2) d\mu_{n-d_1} f_{H_1}(\alpha)\alpha^{n-1}d\alpha.
\end{align*}
Using the fact that $\mathrm{vol}(\SO(n,\R)) = \prod_{i=2}^n iV(i)$, one finds that the product of the terms on the first line here equals $na(n,d_1)$. By Theorem \ref{thm:main}, the integral part is equal to
\begin{align*}
&a(n-d_1, d_2 - d_1)\int_0^{H_1} \alpha^{n-1} \cdot \alpha^{d_2-d_1} \left(\frac{H_2}{\alpha}\right)^{n-d_1} d\alpha \\
&= a(n-d_1, d_2 - d_1)H_2^{n-d_1} \int_0^{H_1} \alpha^{d_2-1}d\alpha \\
&= \frac{a(n-d_1, d_2 - d_1)}{d_2}H_1^{d_2}H_2^{n-d_1}.
\end{align*}
This proves the $k=2$ case. For general $k$, we proceed by induction: we have
\begin{align*}
&\frac{n}{d_1(n-d_1)} \cdot \frac{\tau(n)}{\tau(d_1)\tau(n-d_1)} \cdot \frac{1}{2}\mathrm{vol}\left(\frac{\SO(n,\R)}{\SO(n-d_1,\R) \times \SO(d_1,\R)}\right) \cdot \\
&\int \int_{X_{d_1}}d\mu_{d_1} \int_{X_{n-d_1}} \sum_{\bar{A}_2, \ldots, \bar{A}_k} f_{H_2/\alpha}(\bar{A}_2) \ldots f_{H_k/\det A_{k-1}\alpha}(\bar{A}_k) d\mu_{n-d_1} f_{H_1}(\alpha)\alpha^{n-1}d\alpha.
\end{align*}
The first line is exactly the same as in the $k=2$ case. As for the integral, writing $\mathfrak{d}' = (d_2-d_1, \ldots, d_k-d_1)$, and using the induction hypothesis, we find that it is equal to
\begin{align*}
&a(n-d_1, \mathfrak{d}')\int_0^{H_1} \alpha^{n-1} \cdot \alpha^{d_k-d_1} \prod_{i=2}^{k}\left(\frac{H_i}{\alpha}\right)^{d_{i+1}-d_{i-1}} d\alpha \\
&= a(n-d_1, \mathfrak{d}')\prod_{i=2}^{k}{H_i}^{d_{i+1}-d_{i-1}}\int_0^{H_1} \alpha^{d_2-1} d\alpha \\
&= \frac{a(n-d_1, \mathfrak{d}')}{d_2}\prod_{i=1}^{k}{H_i}^{d_{i+1}-d_{i-1}}.
\end{align*}
Since
\begin{equation*}
a(n-d_1, \mathfrak{d}') = a(n-d_1, d_2-d_1) \prod_{i=2}^{k-1} \frac{n-d_{i-1}}{d_{i+1}-d_{i-1}}a(n-d_i, d_{i+1} - d_i),
\end{equation*}
this gives the desired result.


\begin{thebibliography}{99}

\bibitem{BSW} S. Bai, D. Stehl\'e, W. Wen. Measuring, simulating and exploiting the head concavity phenomenon in BKZ. \emph{Advances in cryptology --- ASIACRYPT 2018}. Part I, 369-404, Lecture Notes in Comput. Sci., 11272, Springer, Cham, 2018. 

\bibitem{DM} S. G. Dani and G. Margulis. Limit distribution of orbits of unipotent flows and values of quadratic forms. Advances in Soviet Math., Vol 16, 1993, pp. 91-137.

\bibitem{FMT} J. Franke, Y. Manin, and Y. Tschinkel. Rational points of bounded height on Fano varieties. Invent. Math. 95 (1989), no. 2, 421-435. 

\bibitem{Kim} S. Kim. Random lattice vectors in a set of size $O(n)$. Int. Math. Res. Not. (2020), 2020(5): 1385-1416.

\bibitem{Kim2} S. Kim. Counting rational points of a Grassmannian. arXiv: 1908.01245

\bibitem{KV} S. Kim and A. Venkatesh. The behavior of random reduced bases. Int. Math. Res. Not. (2018), 2018(20): 6442-6480.

\bibitem{Rogers} C. A. Rogers, Mean values over the space of lattices. Acta Math. 94 (1955), 249-287.

\bibitem{Sch} W. M. Schmidt. Asymptotic formulae for point lattices of bounded determinant and subspaces of bounded height. Duke Math. J. 35 (1968), 327-339.

\bibitem{Sch3} W. M. Schmidt. Masstheorie in der Geometrie der Zahlen. Acta Math. 102, no. 3-4 (1959): 159-224.

\bibitem{SE94} P. Schnorr and M. Euchner. Lattice basis reduction: improved practical algorithms and solving subset sum problems. \emph{Math. Programming} 66 (1994), no. 2, Ser. A, 181-199. 

\bibitem{SW} U. Shapira and B. Weiss. A volume estimate for the set of stable lattices. Comptes Rendus Math\'ematique 352, no.11 (2014), pp.875-879.

\bibitem{Shi} G. Shimura. Introduction to the arithmetic theory of automorphic functions. Princeton University Press, Princeton, N.J., 1971.

\bibitem{Sie} C. L. Siegel. A mean value theorem in geometry of numbers. Ann. of Math. (2) 46, (1945). 340-347. 

\bibitem{Sod} A. S\"odergren, On the Poisson distribution of lengths of lattice vectors in a random lattice. Math. Z. 269 (2011), 945-954.

\bibitem{SS} A. S\"odergren and A. Str\"ombergsson. On the generalized circle problem for a random lattice in large dimension. Adv. Math. 345 (2019), 1042-1074.

\bibitem{Thu2} J. L. Thunder. Asymptotic estimates for rational points of bounded height on flag varieties. Compositio Math. 88 (1993), no. 2, 155-186.

\bibitem{Thu3} J. L. Thunder. Higher-dimensional analogs of Hermite's constant. Michigan Mathematics Journal 45, no. 2 (1998): 301-314.


\end{thebibliography}
\end{document}